\documentclass{amsart}

\usepackage{amsmath,latexsym,amscd,amsbsy,amssymb,amsfonts,amsthm,fleqn,leqno,
euscript, graphicx, color}
\usepackage{color}

\newtheorem{thm}{Theorem}[section]
\newtheorem{pro}[thm]{Proposition}
\newtheorem{lem}[thm]{Lemma}
\newtheorem{cor}[thm]{Corollary}

\newtheorem{que}[thm]{Question}

\newtheorem{rem}[thm]{Remark}

\newcommand{\Int}{\mbox{{\rm int}}\,}
\newcommand{\cl}{\mbox{{\rm cl}}}

\newcommand{\st}{\mbox{{\rm St}}}

\newcommand{\id}{\mbox{{\rm id}}}
\newcommand{\Hom}{\mbox{{\rm Hom}}}

\renewcommand{\int}{{\rm int}}

\begin{document}

\title{Coset spaces of metrizable   groups}

 \thanks{This research was partially supported by  grants IN-115717, PAPIIT (UNAM) and RFBR, 17-51-18051 Bulg\_a}

\author{S.~Antonyan}
\address{Departamento de Matem\' aticas, Facultad de Ciencias, Universidad Nacional Autonoma de M\' exico, Ciudad de M\'exico (Mexico)
}
\email{antonyan@unam.mx}

\author{N.~Antonyan}
\address{Departamento de Matem\' aticas, Escuela de Ingenieria y Arquitectura, Instituto Tecnol\' ogico y de Estudios Superiores de Monterrey,
Campus Ciudad de M\' exico (Mexico)}
\email{nantonya@itesm.mx}

\author{K.\,L.~Kozlov}
\address{Faculty of mechanics and mathematics, Moscow State University}
\email{kkozlov@mech.math.msu.su}


\maketitle


\begin{abstract}
We characterize coset spaces of topological groups which are coset spaces of (separable)
metrizable groups and complete metrizable (Polish) groups.
Besides, it  is shown that for a $G$-space $X$ with a $d$-open action there is a
 topological group $H$ of weight and cardinality less than  or equal to  the weight of $X$ such that $H$ admits a  $d$-open action on $X$.  This is further applied to show that if $X$ is a separable metrizable coset space
then $X$ has a Polish extension which is a coset space of a Polish
group.
\end{abstract}

\bigskip




\section{Introduction}

In the study of topological homogeneity it is natural to ask about groups (or their classes)
which realize the homogeneity in question.
 This approach may give an additional information about the phase space, its degree of homogeneity and   type of the action.
 The survey of A.\,V.~Arhangel'skii and J.~van
Mill~\cite{ArM} can serve as a good introduction to the subject of
the paper.

\medskip

{\bf General Question.} Let $X$ be a coset space of a topological group and assume that $X$ belongs to a   class $\mathcal P$ of topological spaces. Can $X$  be a coset space of a topological group $G$ from a given
class $\mathcal P'$?  We can ask also about additional
properties of the natural action $G\curvearrowright X=G/H$.

\medskip

If a coset space $X$ is locally compact, then $X$ is a coset space
of the group $G=\Hom (X)$ in the $g$-topology~\cite[Theorems 3]{Ar}
which is a Raikov complete topological group~\cite[Theorem 6]{Ar}
and $w(G)\leq w(X)$. The inequality follows from the following
facts: $\Hom (X)$ is a subgroup of $\Hom (\alpha X)$, where $\alpha
X$ is the Alexandroff compactification of $X$; the $g$-topology of
$\Hom (X)$ is induced by the compact open topology on $\Hom (\alpha
X)$; and for the base of the compact open topology on $\Hom (\alpha
X)$  which weight is equal to $w(\alpha X)=w(X)$ one can take open
sets as the sets from a big base of $\alpha X$  and compact sets as
their closures.

From Effros's theorem~\cite{Eff} and the above fact
G.~Ungar~\cite{Ung} deduced that a homogeneous separable metrizable
locally compact space is a coset space of a Polish group.
F.~Ancel~\cite[Question 3]{An} asked whether every homogeneous
Polish space is a coset space, preferably, of some Polish group.
J.~van Mill~\cite{vM7} gave an example of a homogeneous Polish space
which need not be a coset space at all. However the following
question still remains open.
\medskip

{\bf Question.} Is a separable metrizable (respectively Polish) coset space $X$ a
coset space of some separable metrizable (respectively Polish)
group?

\medskip

This question has a positive answer in the case of strongly locally
homogeneous spaces. Recall that a space $X$ is {\it strongly locally
homogeneous} (abbreviated, SLH) if it has an open base $\bf B$ such
that for every $B\in\bf B$ and any $x, y\in B$ there is a
homeomorphism $f: X\to X$ which is supported on $B$ (that is, $f$ is the
identity outside $B$) and moves $x$ to $y$~\cite{Ford}. Any
homogeneous SLH space is a coset space~\cite{Ford}. J.~van Mill made
this result more precise by showing that a separable metrizable
(respectively Polish) SLH space is a coset space of a separable
metrizable~\cite{vM4} (respectively Polish~\cite{vM7}) group.

In this paper we are specially interested in the question when a
metrizable coset space is a coset space of a metrizable group.

Recall that a space $X$ is {\it metrically homogeneous} if there
exists a metric $\rho$ on $X$ such that the action of the group of
isometries ${\rm Iso} (X, \rho)$ is transitive. From~\cite{Kr} it
follows that if a coset space $G/H$ with respect of a compact
subgroup $H\subset G$ is  metrizable, then it  is metrically
homogeneous. The group ${\rm Iso} (X)$ of a metric compactum $X$ is
a closed and equicontinuous subgroup of $\Hom (X)$ in the compact
open topology. Hence, ${\rm Iso} (X)$ in the compact open topology
(which coincides with the topology of pointwise
convergence~\cite[Theorem 1, Ch. X, \S 2, 4]{Burb3}) is compact due
to  Ascoli's theorem.  Therefore, a metrically homogeneous compactum
is a coset space of a compact metrizable group. We proved the
following N.~Okromeshko's theorem.

\begin{thm}\cite{Okr}\label{t1.1}
A metrizable compactum $X$ is a coset space of a compact metrizable
group iff it is metrically homogeneous.
\end{thm}

\medskip

Recall that a metric $\rho$ on a space $X$ is called {\it proper} if
every ball has compact closure. An action of a $G$-space $(G, X,
\alpha)$ is called {\it proper} \cite[Ch., IV]{Bour} if the map $G\times X\to
X\times X$, $(g, x)\to (gx, x)$ is perfect. From~\cite[Theorem
1.1]{AMN} (necessity) and~\cite[Corollary 5.6]{GK} (sufficiency) we
get  the following generalization of Theorem~\ref{t1.1}.

\begin{thm}\label{t1.2}
A locally compact separable metrizable space $X$ is a coset space of
a locally compact separable metrizable group $G$ and the natural
action $G\curvearrowright X=G/H$ is proper iff $X$ is metrically
homogeneous with respect to a proper metric.
\end{thm}

\begin{rem}
In~\cite[Proposition 2.7]{AD} it is shown, in fact, that the
existence of a proper $G$-invariant metric on a coset space $X$ of a
locally compact separable metrizable group $G$ is equivalent to the
fact that $X$ is metrically homogeneous and the stabilizers of the
natural action ${\rm Iso} (X)\curvearrowright X$  are compact.
\end{rem}

A homeomorphism $g\in\Hom (X)$ of a metric space $(X, \rho)$ is a
{\it Lipschitz homeomorphism} if there is $\lambda\geq 1$ such that
$$\lambda^{-1}\rho(x, y)\leq\rho(g(x), g(y))\leq\lambda\rho(x, y)\ \mbox{for all}\ x, y\in X.$$
Let $L(X, \rho)$ denote the group of all Lipschitz homeomorphisms.

A metrizable space $X$ is called {\it Lipschitz homogeneous} if
there exists a metric $\rho$ on $X$ such that the action of $L(X,
\rho)$ in the discrete topology is transitive. A.~Hohti and
H.~Junnila~\cite{HJ} showed that every homogeneous locally compact
separable metrizable space $X$ is Lipschitz homogeneous.
A.~Hohti~\cite[Theorem 3.1]{Hoh} strengthened this result in a
compact case showing that a compact metrizable $X$ is a coset space
of $L(X, \rho)$ in the compact open topology for some metric $\rho$
on $X$.

These results  show that the approach of characterizing groups
which realize homogeneity of metrizable spaces using metrics and
properties of homeomorphisms connected with them is not new.

In Theorem~\ref{tA.4} we characterize coset spaces which are coset
spaces of (separable) metrizable groups. We use (totally bounded)
metrics on coset spaces and the topology of uniform convergence on
the groups of {\it uniform equivalences} with respect to these
metrics.

In Theorem~\ref{tA.5} and Corollary~\ref{cA.4} we characterize
complete metrizable coset spaces which are coset space of complete
metrizable groups (in particular of the same weight). The approach
is based on the possibility of extension of the action $G\times X\to
X$ to an action $\rho G\times \beta_G X\to \beta_GX$, where $\rho G$
is the Raikov completion of the acting group $G$ and $\beta_GX$ is
the extension of $X$ with respect to the maximal totally bounded
equiuniformity. Emphasizing invariant subset of points where the
action is $d$-open, we use a theorem~\cite[Theorem 3]{K4} that a
$d$-open action of a \v Cech complete group (hence of a complete
metrizable group) is open. Remark~\ref{rA.2} shows how this approach
works in the case of Polish SLH spaces.

In Theorem~\ref{t3.1} (and Corollary~\ref{c3.1}) we show that for
any $d$-open action $G\curvearrowright X$ there is a $d$-open action
$H\curvearrowright X$ of a group $H$  of cardinality and weight less
than  or equal  to  the weight of $X$. This may provide some
technique for possible approach to the questions formulated in the
paper. For example, Theorem~\ref{thmadd} shows that if $X$ is a
coset space of a complete metrizable group then it is a coset space
of a complete metrizable group $G$ with  $w(G)=w(X)$. This technique
also allows to construct equivariant compactifications preserving
the weight and the $d$-openness of actions (Theorem~\ref{t3.2}).
Besides, we  show that if $X$ is a separable metrizable coset space,
then $X$ has a Polish extension which is a coset space of a Polish
group. Moreover, to this Polish extension an arbitrary countable
family of homeomorphisms of $X$ can be extended
(Corollary~\ref{c3.4}).


\section{Preliminaries}\label{Prelim}

All spaces, unless otherwise stipulated, are assumed to be
Tychonoff, maps are continuous, notations and  terminology  are from~\cite{E}. A topological space (a space $X$
with topology $\tau$) is denoted by $(X, \tau)$ if we want to
emphasize which topology is considered. Uniformities are introduced
by  families  of covers. For information about uniform spaces
see~\cite{Is}, \cite{RD}.

Nbd(s) is an abbreviation of open neighborhood(s), unless otherwise
stipulated, $\cl_\tau A$ or simply $\cl A$, and $\Int_\tau A$ or
simply $\Int A$ are the closure and the interior of a subset $A$ of
a space $(X, \tau)$ or $X$, respectively. For a family $u$ of subsets
of $X$, and $Y\subset X$ we denote $\st(Y, u)=\bigcup\{U\in u\mid  Y\cap
U\ne\emptyset\}$. For covers $u$ and $v$ of $X$ we denote $u\succ
v$, $u\ast\succ v$ if $u$  refines, respectively star-refines
$v$. A metrizable space is called Polish if it is separable and has
a complete metric. A map $\id$ is an identity map.

For information about topological groups see~\cite{AT}
and~\cite{RD}. By $N_G(e)$ we denote the family of nbds of the unity
$e\in G$. $\rho G$ is the {\it Raikov completion} (or completion in
the two-sided uniformity) of a topological group. $\Hom (X)$ is the
group of homeomorphisms of $X$.

For general information about actions see~\cite{Vr}, \cite{AS},
\cite{Mg2007} and~\cite{K2013}. By $(G, X, \alpha)$ we denote an
action $\alpha$ of a group $G$ on a topological space $X$ such that
the map $\alpha^g: X\to X$, $\alpha^g(x)=\alpha(g, x)$, is
continuous (is a homeomorphism) for every  $g\in G$. Here the action
doesn't depend on the topology of the acting group. If $G$ is a
topological group and the action $\alpha: G\times X\to X$ is
continuous then $(G, X, \alpha)$ is called a {\it $G$-space} or a
{\it topological transformation group}. By a {\it Polish $G$-space}
(respectively, {\it separable metrizable $G$-space}) we mean a
$G$-space $(G, X, \alpha)$ where $X$ and $G$ are Polish
(respectively, separable metrizable) spaces.

If $X=G/H$ is a coset space of a topological group  $G$, then $(G,
X, \alpha)$ is a $G$-space with the  natural action $\alpha(g,
xH)=(gx)H$, where $g, x\in G$.

A pair of maps $(\varphi: G\to H,\ f: X\to Y)$ of $(G, X, \alpha_G)$
to $(H, Y, \alpha_H)$ such that $\varphi: G\to H$ is a homomorphism
and the diagram
$$\begin{array}{ccc}
G\times X &\stackrel{\varphi\times f}\longrightarrow & H\times Y \\
\downarrow\lefteqn{\alpha_G} && \downarrow\lefteqn{\alpha_H} \\
X &\stackrel{f}\longrightarrow & Y
\end{array}$$
is commutative is called {\it equivariant} (thus the equivariantness
is defined not only for maps of $G$-spaces with the same acting group; it generalizes the
traditional concept of equivariantness). We shall use the notation $(\varphi, f):
(G, X, \alpha_G)\to (H, Y, \alpha_H)$ in this case.

The commutativity of the
diagram may be written as the fulfillment of the following condition
$$f(gx)=\varphi(g)f(x)\ \mbox{for any}\ x\in X,\ g\in G.$$
The composition of eqiuvariant pairs of maps is an equivariant pair.
If $f$ is an embedding then the pair $(\varphi, f)$ is an {\it
equivariant embedding} of $(G, X, \alpha_G)$ into $(H, Y,
\alpha_H)$.

A $G$-space $(G, X, \alpha)$ is called {\it $G$-Tychonoff}, if there
is an equivariant embedding $(\id, f)$ of $(G, X, \alpha)$ into a
$G$-space $(G, bX, \tilde\alpha)$ where $bX$ is a compactification
of $X$. The {\it maximal $G$-compactification} is denoted $(G,
\beta_GX, \alpha_\beta)$. The set $\{g\in G \mid  gx=x, \forall x\in
X\}$ is the {\it kernel of the action}. If the kernel of an action
is $\{e\}$, then the action is called {\it effective}.

Let $(G, X, \alpha)$ be a $G$-space and $A(X)$ be the set of all
functions  $X\to \mathbb R$. We can define an action of $G$ on
$A(X)$ by $(gf)(x)=f(g^{-1}x)$.

A continuous function $f: X\to\mathbb R$ on a $G$-space $(G, X,
\alpha)$ is called $G$-uniform if for any $\varepsilon>0$ there is
$O\in N_G(e)$ such that $|f(x)-f(gx)|<\varepsilon$ for any $x\in X$,
$g\in O$. $G$-uniform maps were  introduced by J.~de
Vries~\cite{dvr:76} under the name of $\alpha$-uniform maps, where
$\alpha: G\times X\to X$ is an action.

The set $C^*_G(X)$ of all bounded $G$-uniform functions on
$X$ is an invariant subset of $A(X)$. Moreover, $(G, X, \alpha)$ is
$G$-Tychonoff iff $C^*_G(X)$ separates points and closed subsets of
$X$~\cite{AS}.

For an action $\alpha: G\times X\to X$ on a uniform space  $(X,
{\mathcal U})$ the following conditions are equivalent:
\begin{itemize}
\item[(1)] for any $u\in {\mathcal U}$ there exists $O\in N_G(e)$ such that the cover $\{Ox\mid x\in X\}=\{\alpha (O, x)\mid  x\in X\}$  refines  $u$,
\item[(2)] for any $u\in {\mathcal U}$ there are $O\in N_G(e)$ and $v\in {\mathcal U}$ such that $Ov=\{\alpha (O, V)\mid  V\in v\}\succ u$.
\end{itemize}
An action satisfying either condition (1) or (2) is  called {\it
bounded}~\cite[Ch.\,III, \S 7.3]{Vr}. The action is bounded with
respect to some uniformity on $X$ iff $(G, X, \alpha)$ is
$G$-Tychonoff. Uniformity ${\mathcal U}$ on a space $X$ of a
$G$-space $(G, X, \alpha)$ is called an {\it
equiuniformity}~\cite{Mg1} if it is {\it saturated} (i.e. that all
homeomorphisms from $G$ are uniformly continuous) and the action is
bounded with respect to ${\mathcal U}$.  An action $\alpha: G\times
X\to X$ on a uniform space  $(X, {\mathcal U})$ equipped with an
equiuniformity $\mathcal U$, has a continuous extension
$\tilde\alpha_\rho: \rho G\times \tilde X\to\tilde X$, where  $\rho
G$ is the Raikov completion of $G$ and $ \tilde X $ is  the
completion  of $X$ with respect to ${\mathcal U}$~\cite{Mg1}.
Evidently, the restriction $\tilde\alpha: G\times \tilde X\to\tilde
X$ of $\tilde\alpha_\rho$ is also continuous. Among equiuniformities
there is the maximal one.

An action  $\alpha: G\times X\to X$ is called
\begin{itemize}
\item[] {\it open} if $x\in\int (Ox)$ for any point $x\in X$ and  $O\in N_G(e)$;
\item[] {\it $d$-open} if $x\in\int (\cl (Ox))$ for any point $x\in X$ and  $O\in
N_G(e)$.
\end{itemize}
If $(G, X, \alpha)$ is a $G$-space with a $d$-open action, then  $X$
is a direct sum of clopen subsets ({\it components of the action}).
Each component of the action is the closure of the orbit of an
arbitrary point of this component~\cite[Remark 2]{CK3}. If $(G, X,
\alpha)$ is a $G$-space with an open action and one component of
action, then $X$ is a coset space of $G$. Further information about
($d$-)open actions,  their properties and natural uniform structures
that are induced by ($d$-)open actions can be found in~\cite{CK3},
\cite{K2013}, \cite{K4} and \cite{K2016}.


\section{Compactification theorem}

The following result is a part of~\cite[Theorem 2.13]{Mg3}. We shall
present the proof for  completeness of the  exposition. By $ib(G)$
will be denoted  the index of narrowness of a topological group $G$,
see~\cite[Chapter 5, \S 5.2]{AT}.

\begin{lem}\label{lA.2}
Let $(G, X, \alpha)$ be a $G$-space and $T$ be an invariant subset
of $C^*_G(X)$ that separates points and closed subsets of $X$. Then
the initial uniformity ${\mathcal U}_T$ with respect to $T$ is a totally bounded equiuniformity on $X$.
\end{lem}

\begin{proof}
The uniformity ${\mathcal U}_T$ is compatible with the topology of
$X$ since the maps from $T$ are continuous and separate points and
closed subsets of $X$. Further,  ${\mathcal U}_T$  is totally
bounded since it is the initial uniformity with respect to the maps
to compacta and saturated since $T$ is an invariant subset of
$C^*_G(X)$. Finally, since $T$ consists of $G$-uniform functions,
the action is bounded by ${\mathcal U}_T$.

\end{proof}

Let $(G, X, \alpha)$ be a $G$-space. The set $C^*_G(X)$ of all
bounded $G$-uniform functions $X\to \mathbb R$ endowed with
point-wise defined algebraic operations and with the sup-norm $\Vert
f\Vert=\sup_{x\in X} f(x)$ is a Banach space. It is well-known that
the action $\alpha^*:G\times C^*_G(X)\to C^*_G(X)$,
$$\alpha^*(g,
f)(x)=f(\alpha(g^{-1}, x)), \ x\in X,$$ is linear, isometric and
continuous (see, e.g., ~\cite{AS}).

\begin{lem}\label{lA.3}
Let $(G, X, \alpha)$ be a $G$-space, $f\in C^*_G(X)$ and $T_f=Gf$.
Then the initial pseudouniformity ${\mathcal U}_{T_f}$ with respect
to the maps from $T_f$ has weight $\leq ib(G)$.
\end{lem}

\begin{proof} For $f\in C^*_G(X)$ let $O_n\in N_G(e)$ be such that
$|f(x)-f(gx)|<\frac{1}{n}$ for $x\in X$, $g\in O_n$, $n\in\mathbb
N$. For every $n\in\mathbb N$ there exists $G_n\subset G$,
$|G_n|\leq ib(G)$ such that $G_nO_n=G$. Put
$G_\infty=\bigcup\{G_n\mid n\in\mathbb N\}$. Then the
cardinality of $G_\infty f$ is $ib(G)$.

To finish the proof it is enough to show that the initial
pseudouniformity with respect to the maps from $G_\infty f$
coincides with ${\mathcal U}_{T_f}$.

The subbase of $u\in {\mathcal U}_{T_f}$ is of the form
$(hf)^{-1}v$, where $v$ is a uniform cover of $(hf)(X)\subset\mathbb
R$ by intervals of diameter $\varepsilon$, $\varepsilon>0$, $h\in
G$. Let $O_n\in N_G(e)$, $O_n^{-1}=O_n$, be such that
$|f(x)-f(gx)|<\frac{1}{n}<\frac{\varepsilon}3$ for $x\in X$, $g\in
O_n$. Take $g_h\in G_\infty$ such that $h\in g_hO_n$ and a uniform
cover $w$ of $(g_hf)(X)\subset\mathbb R$ by intervals of diameter
$\frac{\varepsilon}3$. Then $(g_hf)^{-1}w\succ (hf)^{-1}v$. In fact,
for $x, y\in W\in (g_hf)^{-1}w$ we have
$|(g_hf)(x)-(g_hf)(y)|<\frac{\varepsilon}3$. Then
$$|(hf)(x)-(hf)(y)|\leq
|(hf)(x)-(g_hf)(x)|+|(g_hf)(x)-(g_hf)(y)|+|(g_hf)(y)-(hf)(y)|=$$
$$=|f(h^{-1}x)-f(g_h^{-1}x)|+|(g_hf)(x)-(g_hf)(y)|+|f(g_h^{-1}y)-f(h^{-1}y)|=$$
$$=|f(t_h(g_h^{-1}x))-f(g_h^{-1}x)|+|(g_hf)(x)-(g_hf)(y)|+|f(g_h^{-1}y)-f(t_h(g_h^{-1}x))|\leq$$
$$\frac{\varepsilon}3+\frac{\varepsilon}3+\frac{\varepsilon}3=\varepsilon,$$
where $t_h\in O_n$.

Thus, the pseudouniformity with respect to the maps from $G_\infty
f$ is a base of ${\mathcal U}_{T_f}$.
\end{proof}

\begin{thm}\cite{Mg3}\label{tA.3}
For a $G$-Tychonoff space $(G, X, \alpha)$ there is an equivariant
embedding $(\id, f):(G, X, \alpha)\to (G, B, \tilde\alpha)$ where
$B$ is compact and $w(B)\leq ib(G)\cdot w(X)$.
\end{thm}

\begin{proof} There exists a subset $T'\subset C^*_G(X)$ that separates points and
closed subsets of $X$ such that $|T'|\leq w(X)$. Therefore, the
invariant subset $T=GT'=\{gf\mid g\in G, f\in T'\}\subset C^*_G(X)$
separates points and closed subsets of $X$.

By Lemma~\ref{lA.2} ${\mathcal U}_T$ is a totally bounded
equiuniformity on $X$. Hence, the completion of $X$ with respect to
${\mathcal U}_T$ is a compactum $B$ and $(G, B, \tilde\alpha)$ is a
$G$-compactification of $(G, X, \alpha)$ (see Section \ref{Prelim}).

Since $w({\mathcal U}_T)\leq |T'|\cdot\sup\{w({\mathcal
U}_{T_f})\mid  f\in T'\}$, by Lemma~\ref{lA.3} $w({\mathcal
U}_T)\leq |T'|\cdot ib(G)\leq ib(G)\cdot w(X)$. Taking into account
that the weight of a compactum coincides with the weight of its
unique uniformity, we get $w(B)\leq ib(G)\cdot w(X)$.
\end{proof}

A group $G$ is $\omega$-narrow iff $ib(G)\leq\aleph_0$,
see~\cite[Chapter 5, \S 5.2]{AT}. The following corollary is a
compactification theorem for actions of $\omega$-narrow groups.

\begin{cor}\cite{Ant}\label{c2.9}
For a $G$-Tychonoff space $(G, X, \alpha)$ where $G$ is
$\omega$-narrow there is an equivariant embedding $(\id, f):(G, X,
\alpha)\to (G, B, \tilde\alpha)$ where $B$ is compact and $w(B)\leq
w(X)$.
\end{cor}


\section{Characterizations of coset spaces of metrizable groups}

\begin{thm}\cite[Theorem 2.2]{Ford}\label{tA.1}  If $X$ possesses
a uniform structure $\mathcal U$ under which every element of some
homeomorphism group $G$ is uniformly continuous, then $G$ is a
topological group relative to the uniform convergence topology induced
by the uniformity.
\end{thm}

\begin{rem}\label{rA.1}
{\rm (a) The sets $O_u=\{g\in G\mid \forall x\in X \ g(x)\in\st(x,
u)\}$, $u\in \mathcal U$, constitute  a base of nbds (not open) of
the unit element $e$.

(b) The uniformity of uniform convergence in Theorem~\ref{tA.1}
coincides with the right uniformity on $G$~\cite[Chapter 2, Exercise
2]{RD}\label{rA.1}.

(c) We can add the following to Theorem~\ref{tA.1}. Uniformity
$\mathcal U$ is an equiuniformity with respect to  the natural
action $G\curvearrowright X$ and, hence, $(G, X, \alpha)$ is a
$G$-Tychonoff space. In fact, each element of $G$ is a uniformly
continuous map. Therefore, the uniformity $\mathcal U$ is saturated.

For arbitrary $u\in\mathcal U$ take $v\in\mathcal U$ such that
$v\ast\succ u$ and a nbd $O_v$ of $e$ in $G$. Then $O_v v\succ u$.
Hence, $\mathcal U$ is bounded.}
\end{rem}

Recall that a  bijective map $f$ of uniform spaces is called a {\it uniform
equivalence} if $f$ and $f^{-1}$  are
uniformly continuous. We denote by $H(X, {\mathcal U})$  the  group of uniform
equivalences of a uniform space $(X, {\mathcal U})$. An evident
consequence of Remark~\ref{rA.1}(a) is the following corollary.

\begin{cor}\label{cA.1}
If $w(\mathcal U)\leq\aleph_0$ for $(X, {\mathcal U})$, then  $H(X,
{\mathcal U})$ endowed with  the topology of uniform convergence is metrizable.
\end{cor}

It is well known that there is a unique uniformity  $\mathcal U$ an a compactum $X$ and
$\Hom (X)=H(X, {\mathcal U})$.

\begin{pro}\label{pA.1}
If $X$ is a compactum then the topology of uniform convergence on
any subgroup $G$ of $\Hom (X)$ coincides with the compact open
topology.
\end{pro}

\begin{proof}
Let $\mathcal U$ be the unique uniformity on $X$, $\tau_{\mathcal
U}$ and $\tau_{CO}$ be the topology of uniform convergence and the
compact open topology on $G$, respectively.

In order to prove the inclusion $\tau_{\mathcal U}\supset\tau_{CO}$
it is enough to show that any nbd in compact open topology of the
unity  $e\in G$ contains a nbd of  $e$ in the topology of uniform
convergence. A  typical  sub-basic nbd of  $e$ in the topology
$\tau_{CO}$  looks like  $[K, O]=\{g\in G \mid g(K)\subset O\}$,
where $O$ is an open subset of $X$, $K\subset O$ and $K$ is compact.
There is $u\in\mathcal U$ such that $\st (K, u)\subset O$. Then
$O_u\subset [K, O]$.

For the proof of the converse inclusion take arbitrary nbd of  $e$
in the topology $\tau_{\mathcal U}$ of the form $O_u$, $u\in\mathcal
U$. Let $w, v\in\mathcal U$ be such that $v\ast\succ w\ast\succ u$.
Take a finite subcover  $\{V_1,\dots, V_n\}\subset v$ of $X$ and let
$\st (V_k, v)\subset W_k$, $W_k\in w$, $k=1,\dots, n$. Then
$W=\bigcap\{[\cl (V_k), W_k] \mid  k=1,\dots, n\}$ is a nbd of $e$
in the topology $\tau_{CO}$ and $W\subset O_u$.
\end{proof}

\begin{cor}\label{cA.2}
If $w(\mathcal U)\leq\aleph_0$ for a totally bounded uniformity
$\mathcal U$ on $X$ then $H(X, {\mathcal U})$ in the topology of
uniform convergence is separable and metrizable.
\end{cor}

\begin{proof}
The completion $bX$ of $X$ with respect to $\mathcal U$ is a
metrizable compactum. Since each homeomorphism $g\in H(X, {\mathcal
U})$ is uniformly continuous, it admits a unique  extension $g':
bX\to bX$ which is uniformly continuous with respect to the unique
uniformity on $bX$. Clearly, $(g\circ h)'=g'\circ h'$ and
$1_X'=1_{bX}$. In other words, the map $g\mapsto g'$ is a
monomorphism of the group $H(X, {\mathcal U})$ into the group $\Hom
(bX)$. Thus,  identifying $H(X, {\mathcal U})$  with its image in
$\Hom (bX)$, we can assume that $H(X, {\mathcal U})$ is a subgroup
of $\Hom (bX)$. Then by Proposition~\ref{pA.1} the topology of
uniform convergence on $H(X, {\mathcal U})$ coincides with the
compact open topology. The latter topology is separable and
metrizable. The rest follows from the coincidence of $\mathcal U$
with the restriction on $X$ of the the unique uniformity of $bX$.
\end{proof}

\begin{thm}\label{tA.2}
Let $(G, X, \alpha)$ be a $G$-space with an open effective action,
$\mathcal U$ be the maximal equiuniformity on $X$. Then
\begin{enumerate}
\item each element
of $G$ is a uniform equivalence (with respect to $\mathcal U$),
\item the topology of uniform convergence $\tau_{\mathcal U}$ on $G$ is
coarser than the original one,
\item $((G, \tau_{\mathcal U}), X, \alpha)$
is a $G$-space with an open action,
\item $\mathcal U$ is the maximal equiuniformity on $X$ in the
$G$-Tychonoff space $((G, \tau_{\mathcal U}), X, \alpha)$.
\end{enumerate}
\end{thm}

\begin{proof} (1)  follows  immediately  from the definition of an
equiuniformity.

(2) Let $N_G(e)$ denote the base of nbds of  $e$ in the original
topology. If the action is open then the covers $u=\{Ux \mid  x\in
X\}$, $U\in N_G(e)$, are the base of the maximal equiuniformity
$\mathcal U$ on $X$. Since for $V=V^{-1}$, $V^3\subset U$ we have
$\{Vx\mid  x\in X\}\ast\succ\{Ux \mid  x\in X\}$, the base of nbds
(not open) of  $e$ in  the topology of uniform convergence
$\tau_{\mathcal U}$ consists of the sets $O_u=\{g\in G\mid \forall
x\in X\ g(x)\in Ux\}$,  $U\in N_G(e)$. Hence, it suffices to prove
that for every set $O_u$, its interior $\Int_{\tau_{\mathcal U}}O_u$
is  open in the original topology. In  turn,  it is enough to show
that $U\subset\Int_{\tau_{\mathcal U}} O_u$.

Indeed, for arbitrary $g\in U$, there is $V\in N_G(e)$ such that
$Vg\in U$. For any $h\in O_v$ we have $h(x)\in Vx$ for all $x\in X$.
Thus $hg(x)\in Vg(x)\subset Ux$ for all $x\in X$. Therefore, $hg\in
O_u$ and $O_vg\subset O_u$  showing that $g\in \Int_{\tau_{\mathcal
U}}O_u$. Hence, $U\subset\Int_{\tau_{\mathcal U}} O_u$, as required.

(3) From Remark~\ref{rA.1} (c) it follows that $((G, \tau_{\mathcal U}),
X, \alpha)$ is a $G$-space, and the action is open since the
topology $\tau_{\mathcal U}$ is coarser than the original one.

(4) From the inclusions $U\subset\Int_{\tau_{\mathcal U}} O_u\subset
O_u$ for $U\in N_G(e)$ and the equality $O_ux=Ux$, $x\in X$, it
follows that $(\Int_{\tau_{\mathcal U}} O_u)x=Ux$, $x\in X$. Hence,
the maximal equiuniformity on $X$ for the action of $G$ in the
topology $\tau_{\mathcal U}$ is the same as the original one.
\end{proof}

The following lemma immediately follows from the definition of the
topology of uniform convergence.

\begin{lem}\label{lA.1}
Let $(G, X, \alpha)$ be a $G$-space with equiuniformities ${\mathcal
U}\supset{\mathcal V}$ on $X$. Then for the topologies
$\tau_{\mathcal U}$, $\tau_{\mathcal V}$ of uniform convergence on
$G$ we have $\tau_{\mathcal U}\supset\tau_{\mathcal V}$.
\end{lem}

\begin{thm}\label{tA.4} The following conditions are equivalent for a (separable) metrizable space $X$:
\begin{itemize}
\item[(a)] $X$ is a coset space of a (separable) metrizable group
$G$;
\item[(b)] there is a (totally bounded) metric $\rho$ on $X$ such that $X$ is a coset space of the
group of uniform equivalences with respect to $\rho$ in the topology
of uniform convergence.
\end{itemize}
\end{thm}

\begin{proof}
$(a)\Longrightarrow (b)$ Let a (separable) metrizable space $X$ be
a coset space of a (separable) metrizable group $G$. Then $(G, X,
\alpha)$ is a $G$-space with a natural action of $G$ on $X$ by left
translations (open and transitive).

If $G$ is separable and metrizable, then by Theorem~\ref{tA.3} there
is an equivariant embedding $(\id, f):(G, X, \alpha)\to (G, B,
\tilde\alpha)$ where $B$ is a metrizable compactum. There is a
unique (up to a metric equivalence) totally bounded metric
$\tilde\rho$ on $B$, and all elements of $G$ are uniform
equivalences with respect to  the correspondent unique totally
bounded uniformity ${\mathcal U}_{\tilde\rho}$ with countable weight
which is an equiuniformity. Let $\rho=\tilde\rho|_{X\times X}$. Then
$\rho$ is a totally bounded metric on $X$ with respect to which all
elements of $G$ are uniform equivalences. Moreover, the
correspondent totally bounded equiuniformity  ${\mathcal
U}_\rho={\mathcal U}_{\tilde\rho}|_{X}$ is weaker than the maximal
equinuformity on $X$.

If $G$ is metrizable then the weight of the maximal equiuniformity
${\mathcal U}$ on $X$ is countable. In fact, the covers $\{Ox\mid
x\in X\}$, $O\in N_G(e)$, constitute a  base of the maximal
equiuniformity~\cite{CK3}. Since $G$ is metrizable we can take the
countable base of the maximal equiuniformity. Let $\rho$ be a metric
on $X$ which generates this uniformity.

In both cases, by Theorem~\ref{tA.2} and Lemma~\ref{lA.1}, the
topology $\tau_\rho$ of uniform convergence with respect to $\rho$
on $G$ is weaker than the original topology on $G$. Hence, by
Theorem~\ref{tA.2}, $((G, \tau_{\rho}), X, \alpha)$ is a $G$-space
with an open action. Evidently, the natural action $\alpha': H(X,
{\mathcal U})\times X\to  X$  of the whole group of uniform
equivalences  is open since $(G, \tau_{\rho})$ is a subgroup of
$(H(X, {\mathcal U}), \tau_{\mathcal U})$ and $\alpha'|_{G\times
X}=\alpha$.

$(b)\Longrightarrow (a)$ The group $G$ of uniform equivalences with
the topology of uniform convergence with respect to the (totally
bounded) uniformity with countable weight (which corresponds  to the
(totally bounded) metric) is (separable) metrizable by
Corollary~\ref{cA.1} (Corollary~\ref{cA.2}). Thus, $G$ is
(separable) metrizable whenever  $X$ is so.
\end{proof}

\begin{lem}\label{lA.5} Let $(G, X, \alpha)$ be a $G$-space,
$H$ be a dense subgroup of $G$. Then the set of points $D_G$, where
the action $\alpha$ is $d$-open coincides with the set of points
$D_H$, where the action $\alpha|_{H\times X}$ is $d$-open.
\end{lem}

\begin{proof} Since  $N_H(e)=H\cap N_G(e)$,  for any element $O\in N_G(e)$ the set  $O'=(O\cap H)\in N_H(e)$
is dense in $O$. Therefore, for any point $x\in X$ the set $O'x$ is
dense in $Ox$ and ${\rm cl} (Ox)={\rm cl} (O'x)$. Hence, $x\in{\rm
int} ({\rm cl} (Ox))$ iff $x\in{\rm int} ({\rm cl} (O'x))$.
\end{proof}

\begin{thm}\label{tA.5} The following conditions are equivalent for a complete metrizable space $X$:
\begin{itemize}
\item[(a)] $X$ is a coset space of a complete metrizable group,
\item[(b)] there exist a metrizable group $G$ and a $d$-open action $G\times X\to X$  with only one
component of action such that the set of points in $\beta_GX$ where
the action $G\times \beta_GX\to \beta_GX$ is $d$-open coincides with
$X$.
\end{itemize}
\end{thm}

\begin{proof}
$(a)\Longrightarrow (b)$  Let a complete metrizable space $X$ be a
coset space of a  complete metrizable group $G$. Then $(G, X,
\alpha)$ is a $G$-Tychonoff space with the  natural action of $G$ on
$X$ by left translations which is open and transitive. Hence, the
action $\alpha$ is $d$-open and has one component.

We claim that in the $G$-space $(G,\beta_G X,\alpha_\beta)$, the set
$D$ of points where the action $G\times \beta_GX\to \beta_GX$ is
$d$-open coincides with $X$.  In fact, the inclusion $X\subset D$ is
due to~\cite[Proposition 3]{K4}. By~\cite[Theorem 3]{K4} the
restriction of action of $G$ on $D$ is open and, hence, transitive.
Thus, $D=X$.

$(b)\Longrightarrow (a)$  Let $\alpha_\beta:G\times \beta_GX\to
\beta_GX$ and $\alpha_{\rho\beta}: \rho G\times \beta_GX\to
\beta_GX$ be extensions of the action $\alpha: G\times X\to X$. The
$G$-space $(\rho G, \beta_GX, \alpha_{\rho\beta})$ and the
equivariant embedding $(i, \id): (G, \beta_G X, \alpha_{\beta})\to
(\rho G, \beta_GX, \alpha_{\rho\beta})$, where $i:G\to\rho G$ is the
natural embedding, are well defined by~\cite{Mg1}. The group $\rho
G$ is complete metrizable as the Raikov completion of a metrizable
group~\cite{RD}. By Lemma~\ref{lA.5} the set of the points where the
action $\alpha_{\rho\beta}$ is $d$-open coincides with the set of
the points where the action $\alpha_{\beta}$ is $d$-open. The latter
set coincides with $X$ by~\cite[Proposition 3]{K4}.

Therefore, the $G$-space $(\rho G, X, \alpha_{\rho\beta}|_X)$ with a
$d$-open action and one component of action is well defined. Recall
that by \cite[Theorem 3]{K4}, a $d$-open action of a \v Cech
complete group is open. Further, since every complete metrizable
group is \v Cech complete, we conclude that $\alpha_{\rho\beta}|_X$
is open. But, an open action with one component is transitive.
Hence, $X$ is a coset space of $\rho G$.
\end{proof}

From the proof of Theorem~\ref{tA.5} we have.

\begin{cor}\label{cA.3} Let $G$ be a (separable) metrizable group, and let  $(G, X, \alpha)$ be a $G$-space with a $d$-open action with only  one component of action.
Assume that  the set of  the points in $\beta_GX$ where the action
$G\times \beta_GX\to \beta_GX$ is $d$-open coincides with $X$. Then
$X$ is a complete metrizable (Polish) space.
\end{cor}

Since every metrizable group of infinite weight $\tau$ contains a
dense subgroup of cardinality $\tau$, Theorem~\ref{tA.5}
and~\cite[Lemma 3]{K4} imply the following corollary.

\begin{cor}\label{cA.4}
The following conditions are equivalent for a complete metrizable
space $X$:
\begin{itemize}
\item[(a)] $X$ is a coset space of a complete metrizable group of infinite weight $\tau$,
\item[(b)] there exist a metrizable group $G$ with  $|G|\leq\tau$  and a  $d$-open action $G\times X\to X$  with only one component of action
such that in $\beta_GX$ the set of points where the action $G\times \beta_GX\to \beta_GX$ is $d$-open coincides with $X$.
\end{itemize}
\end{cor}

\begin{rem}{\rm  According to~\cite{Mg1}, any  action $G\times X\to X$ on a $G$-Tychonoff space can be extended to an action $\rho G\times \tilde X\to \tilde X$,
where $\rho G$ is the  the Raikov completion of the acting  group
$G$  and $\tilde X$ is  the completion of the phase space $X$ with
respect to any equiuniformity. Therefore, in Theorem~\ref{tA.5} and
Corollary~\ref{cA.4} we can use the completion of a phase space $X$
with respect to any equiuniformity and not only the maximal totally
bounded equiuniformity which corresponds just to $\beta_GX$.}
\end{rem}

\begin{que}
Can in Theorem~\ref{tA.5} the group $G$ be taken as a subgroup of
$H(X, {\mathcal U})$ with the topology of uniform convergence for
some uniformity $\mathcal U$  with a countable base on $X$?
\end{que}

\begin{rem}\label{rA.2}
{\rm An alternative proof that a Polish SLH space $X$ is a coset
space of a Polish group in~\cite[Theorem 5]{K4} is the following.

a) By~\cite[Lemma 7, Corollary 2]{K4} there exist a metrizable
compactification $bX$ of $X$ and a countable subgroup $T$ of $\Hom
(bX)$ in the compact open topology such that $X$ is an invariant
subset for the action of $T$ on $bX$  and the action is $d$-open at
every point of $X$.

b)  By~\cite[Lemmas 9, 10]{K4} we can strengthen the topology on $T$
in such a way that it remains metrizable, and the set of points
where the action is $d$-open becomes $X$.

c) The required group is the Riakov completion of $T$ in this
topology.}
\end{rem}


\section{Replacing an action but preserving $d$-openness }

\begin{lem}\label{l3.1}
Let $(G, X, \alpha)$ be a $G$-space with a $d$-open action, and $Y$
be a dense subset of $X$. Then for every $O\in N_G(e)$
$$\lambda_O=\{\int (\cl (Oy))\mid y\in Y\}$$ is a cover of $X$.
\end{lem}

\begin{proof} Let $O\in N_G(e)$ be fixed. Take $V\in N_G(e)$
such that $V^3\subset O$ and $V^{-1}=V$. It is enough to show that
$\gamma_V=\{\int (\cl (Vx))\mid x\in X\}$  refines  $\lambda_O$.

For the set $\int (\cl (Vx))\in\gamma_V$ there exists a point $y\in
Y$ such that $y\in\int (\cl (Vx))$ (since $Y$ is dense in $X$).
Lemma 3 of~\cite{CK3} establishes the inclusion $\st (y,
\gamma_V)\subset\int (\cl (Oy))$. Thus,
$\int(\cl(Vx))\subset\int(\cl(Oy))$ and, hence,
$\gamma_V\succ\lambda_O$, as required.
\end{proof}

\begin{pro}\label{p3.1}
Let $(G, X, \alpha)$ be a $G$-space with a $d$-open action. Then
there exist a family $\mathcal O\subset N_G(e)$ and a dense subset
$Y$ of $X$ such that
\begin{itemize}
\item[(1)] $|\mathcal O|\leq w(X)$, $|Y|\leq w(X)$,
\item[(2)] $\gamma_O=\{\int (\cl (Oy))\mid y\in Y\}$,  $O\in\mathcal O$,  is a cover of $X$,
\item[(3)]  $\{\int (\cl (Oy))\mid y\in Y,\ O\in {\mathcal O}\}$ is a base of the
topology of $X$,
\item[(4)]  $\{\int (\cl (Oy))\mid O\in {\mathcal O}\}$ is a local base at the point $y$ for every $y\in Y$.
\end{itemize}
\end{pro}

\begin{proof}
Since the action is continuous, the family $\{\int (\cl (Ox))\mid
x\in X,\ O\in N_G(e)\}$ is a base for $X$. There exists its
subfamily $\mathcal B$ of cardinality $w(X)$ which is a base of $X$.
Each element $B\in\mathcal B$ is of the form $\int (\cl (O_Bx_B))$
where $x_B\in X$, $O_B\in N_G(e)$.

Put $Y=\{ x_B \mid  B\in \mathcal B\}$. Evidently, $Y$ is a dense subset
of $X$ and $|Y|\leq w(X)$. For every $y\in Y$ let ${\mathcal
O}_y\subset N_G(e)$ be such that $|{\mathcal O}_y|\leq w(X)$ and
$\{\int (\cl (Oy))\mid O\in {\mathcal O}_y\}$ is a local  base at the point $y$.

Put $\mathcal O=\{ O_B\mid B\in \mathcal B\}\cup\bigcup\{{\mathcal
O}_y\mid y\in Y\}$.  Evidently, $|\mathcal O|\leq w(X)$. Conditions
(3) and (4) follow from the definitions of $Y$ and $\mathcal O$, and
condition (2) is a consequence of Lemma~\ref{l3.1}.
\end{proof}

\begin{rem}\label{r3.1}
{\rm In Proposition~\ref{p3.1} one can assume that the family
${\mathcal O}$ is a filter basis of the same cardinality such that
\begin{itemize}
\item[{\rm (a)}] all the sets  $O\in \mathcal O$ are symmetric, i.e.
$O=O^{-1}$, and
\item[{\rm (b)}] for every  $O\in\mathcal O$ there is $V\in\mathcal O$
such that $V^2\subset O$.
\end{itemize}
Therefore, $H=\bigcap\{O\mid O\in {\mathcal O}\}$ is a subgroup of
$G$ and, since $Hy=y$ for any $y\in Y$, and $Y$ is dense in $X$, it
follows that $H$ belongs to the kernel of the action. Moreover, $\psi
(H, G)\leq w(X)$, where $\psi (H, G)$ is a pseudocharacter of $H$ in
$G$.}
\end{rem}

The following conditions are sufficient for an action to be $d$-open.

\begin{pro}\label{p3.2}
Let $(G, X, \alpha)$ be a $G$-space and $Y$ be a dense subset of $X$
such that:
\begin{itemize}
\item[{\rm (1)}] the action is $d$-open at points of $Y$;
\item[{\rm (2)}] for any $O\in N_G(e)$ the family $\{\int (\cl
(Oy))\mid y\in Y\}$ is a cover of $X$;
\item[{\rm (3)}] for any $y\in Y$ and $O, U, V\in N_G(e)$ with  $U^{-1}=U$, $U^2\subset O$, one has
$$\cl (Uy)\subset O\int (\cl (Vy))=\bigcup\{g\int (\cl (Vy))\mid g\in O\}.$$
\end{itemize}
Then the action on $X$ is $d$-open.
\end{pro}

\begin{proof}
For arbitrary $O\in N_G(e)$, $O^{-1}=O$, take $U\in N_G(e)$ such
that $U^{-1}=U$, $U^4\subset O$.  By condition (2), for any $x\in X$ there is $y\in Y$
such that $x\in\int (\cl (Uy))$. For any $V\in
N_G(e)$, from condition (3), we have  $x\in\int (\cl (Uy))\subset
U^2\int (\cl (Vy))$. Hence, there is $g\in U^2$ such that $gx\in
\int (\cl (Vy))$ and $Ox\cap\int (\cl (Uy))$ is dense in $\int (\cl
(Uy))$. Therefore, $x\in\int (\cl (Uy))\subset\cl (Ox)$, showing that the
action is $d$-open at $x$.
\end{proof}

\begin{lem}\label{l3.2}
Let $G$ be a group of cardinality $\tau$, and ${\mathcal O}\subset
N_G(e)$ be a  family of cardinality $\tau$. Then there exists a
filter in $N_G(e)$ containing ${\mathcal O}$, with basis ${\mathcal
B}$ of cardinality $\tau$, satisfying the conditions {\rm (a)} and
{\rm (b)} of \ {\rm Remark~\ref{r3.1}} and the following condition:
\begin{itemize}
\item[{\rm (c)}] for any $O\in\mathcal B$ and arbitrary $g\in G$ there is $V\in\mathcal B$
such that $gVg^{-1}\subset O$.
\end{itemize}
\end{lem}

\begin{proof}
Starting with ${\mathcal O}_1={\mathcal O}$ we can

1) assume that the sets $O\in {\mathcal O}_1$ are symmetric (take
the sets $O\cap O^{-1}$ instead of $O$).

Then,

2) for any $O\in {\mathcal O}$ take $O_n=O_n^{-1}$ ($O_1=O$),
$O_{n+1}^2\subset O_n$, $n\in\mathbb N$.

And at last,

3) let ${\mathcal O}'_1$ be the family of finite intersections of
sets $\{O_n\mid O\in {\mathcal O}_1, n\in\mathbb N\}$.

The obtained family ${\mathcal O}'_1$ is a filter basis of
cardinality $\tau$, satisfying the conditions {\rm (a)} and {\rm
(b)} of Remark~{\rm\ref{r3.1}}.

Put

4) ${\mathcal O}_2=\{gOg^{-1}\mid O\in {\mathcal O}'_1, g\in G\}$.

Obviously $|{\mathcal O}_2|=\tau$.

Applying the same procedure to ${\mathcal O}_2$ we obtain the family
${\mathcal O}_3$ of cardinality $\tau$ and so on. The family
$\bigcup\{{\mathcal O}_n\mid n\in\mathbb N\}$ is the required filter
basis ${\mathcal B}\subset N_G(e)$ of cardinality $\tau$.
\end{proof}

\begin{thm}\label{t3.1}
Let $(G, X, \alpha)$ be a $G$-space with a $d$-open action. Then
there exist a group $H$ and a continuous $d$-open action $\gamma:
H\times X\to X$ such that $|H|\leq w(X)$ and  $w(H)\leq w(X)$.
\end{thm}

\begin{proof}
Let the family ${\mathcal O}\subset N_G(e)$ and the dense subset $Y$
of $X$ be taken as in Proposition~\ref{p3.1}. We also may assume
that the family ${\mathcal O}$ is a filter basis satisfying the
conditions (a) and (b) of Remark~\ref{r3.1}. The following procedure
will allow to construct a group of cardinality $w(X)$ and a filter
basis of subsets of $N_G(e)$ of cardinality $w(X)$ which satisfy
the conditions (a), (b) of Remark~{\rm\ref{r3.1}} and  the condition (c) of
Lemma~\ref{l3.2}.

Since $\alpha$ is a $d$-open action, for $y\in Y$ and $O\in\mathcal
O$, we can choose
 a set $G(O,
y)\subset O$ of cardinality $w(X)$ such that $G(O, y)y$ is dense in
$\int (\cl (Oy))$.

\noindent By~\cite[Lemma 4]{CK3} for any pair $O, V\in\mathcal O$
and $x\in X$ we have $$\cl (Ox)\subset O^2\int (\cl (Vx)).$$ There
is a subset $H(O, V, x)\subset O^2$ such that $|H(O, V, x)|\leq
w(X)$ and $\cl (Ox)\subset H(O, V, x)\int (\cl (Vx))$.

Let $H_1$ be the subgroup of $G$ generated by the subset
$$\bigcup\{
H(O, V, y)\mid O, V\in {\mathcal O},\ y\in Y\}\cup\bigcup\{G(O,
y)\mid O\in {\mathcal O},\ y\in Y\}.$$
 Evidently, $|H_1|\leq w(X).$

By Lemma~\ref{l3.2}, for the group $H_1$ and the family ${\mathcal
O}\subset N_G(e)$ there exists a filter in $N_G(e)$ containing
${\mathcal O}$, with basis ${\mathcal B}_1$ of cardinality $w(X)$,
satisfying the conditions (a), (b) of Remark~{\rm\ref{r3.1}} and the condition (c)
of Lemma~\ref{l3.2}.

Applying the same procedure to the family ${\mathcal B}_1\subset
N_G(e)$ we obtain a  subgroup $H_2$, $|H_2|\leq w(X)$, and a filter
in $N_G(e)$ containing ${\mathcal O}$, with basis ${\mathcal B}_2$
of cardinality $w(X)$, satisfying the conditions (a), (b) of
Remark~{\rm\ref{r3.1}} and the condition (c) of Lemma~\ref{l3.2},  and so on. The
subgroup $H'=\bigcup\{H_n\mid n\in\mathbb N\}$ is of cardinality
$w(X)$ and the family ${\mathcal B}=\bigcup\{{\mathcal B}_n\mid
n\in\mathbb N\}$ is the filter basis of cardinality $\tau$
satisfying the conditions (a), (b) of Remark~{\rm\ref{r3.1}} and the condition (c)
of Lemma~\ref{l3.2}.

Let $H$ be the quotient group $H'/T$, where $T=(\bigcap\{O\mid O\in
{\mathcal B\}})\cap H'$ is a normal subgroup of $H'$ and $q: H'\to
H'/T$ be the quotient homomorphism. Following to~\cite[Proposition
1.21]{RD}, we shall consider the topology on $H$ for which the local
basis at the  unity (not necessarily of open sets) is $\{q(O\cap
H')\mid O\in {\mathcal B}\}$. Evidently $|H|\leq w(X)$ and $w(H)\leq
w(X)$.

Since $H'$ is a subgroup of $G$ and $T$ belongs to the kernel of the
action $\alpha$, the action $\gamma: H\times X\to X$ given  by  $\gamma (h,
x)=\alpha (h', x)$, $h'\in q^{-1}(h)$, is well defined. It is
continuous due to the choice of $\mathcal B$ (${\mathcal
O}\subset\mathcal B$) and Proposition~\ref{p3.1}. In order to prove
$d$-openness of $\gamma$ we shall check the fulfillment of
conditions in Proposition~\ref{p3.2}.

The fulfillment of (1). Take $q(O\cap H')$, $\int (q(O\cap H'))\in
N_H(e)$, $O\in {\mathcal B}$, and $y\in Y$. Since the set $G(O, y)y$
is dense in $\int (\cl (Oy))$, and $G(O, y)\subset H'$ we have
$y\in\int (\cl (Oy))\subset\int (\cl (q(O\cap H')y))$.

The fulfillment of (2) is due to Lemma~\ref{l3.1} and the condition
$G(O, y)\subset O$, $\cl (G(O, y)y)\supset\int (\cl (Oy))$, $y\in
Y$, $O\in\mathcal B$.

The fulfillment of (3) follows from the inclusions $\cl (Ox)\subset H(O,
V, x)\int (\cl (Vx))$ and $H(O, V, y)\subset H'$.
\end{proof}

\begin{que}
Do the maximal equiuniformities on $X$ with respect of the actions $\alpha$ and
$\gamma$  coincide ?
\end{que}

\begin{rem}\label{remadd1}
{\rm If in the proof of Theorem~\ref{t3.1} we suppose that the
family ${\mathcal O}$ contains as an element the whole group $G$,
then the components of actions $\alpha$ and $\gamma$ coincide.}
\end{rem}

\begin{rem}\label{remadd}
{\rm If in Theorem~\ref{t3.1}, $\chi (G)\leq w(X)$,
then in the first step of the proof we can take $\mathcal O$,
a local base at $e$,   with
$|{\mathcal O}|\leq  w(X)$.  Then we obtain a
subgroup $H=H_1$ of $G$ with  $w(H)\leq w(X)$, $|H|\leq w(X)$ which acts
$d$-openly on $X$.

In particular, if the group $G$ is metrizable and the phase space
$X$ is separable metrizable then we obtain a metrizable countable
subgroup $H$ of $G$ which acts $d$-openly on $X$.}
\end{rem}

\begin{thm}\label{thmadd}
Let $X$ be a coset space of a complete metrizable group $G$. Then
$X$ is a coset space of a complete metrizable subgroup $H$ of $G$
with $w(H)\leq w(X)$.
\end{thm}

\begin{proof}  By Remark~\ref{remadd}, there is a metrizable subgroup $H'\subset G$ with  $|H'|\leq w(X)$ which acts $d$-openly on $X$.  Its Raikov
completion $H=\rho H'$ is the closure of $H'$ in $G$ and is a
complete metrizable group. Thus, the action $H\times X\to X$ is well
defined and $d$-open. Moreover, from~\cite[Theorem 3]{K4} it follows
that it is open, and hence, $X$ is a coset space of the  complete
metrizable group $H$ with  $w(H)\leq w(X)$.
\end{proof}

\begin{cor}\label{coradd}
Let a Polish space $X$ be a coset space of a complete metrizable
group $G$. Then $X$ is a coset space of a Polish subgroup $H$ of
$G$.
\end{cor}

\begin{que}
Let $X$ be a coset space of a metrizable group $G$. Can it be a
coset space of a metrizable group $H$ with $w(H)\leq w(X)$?
\end{que}

Since ${\rm ib} (G)\leq w (G)$ for a topological group $G$, from
Theorems~\ref{tA.3} and~\ref{t3.1} we get the following result.

\begin{thm}\label{t3.2} For a $G$-space $(G, X, \alpha)$ with a
$d$-open action,  there exist a group $H$  and  an equivariant
compactification  $(H, bX, \tilde\gamma)$  of $(H, X, \gamma)$ such
that:
\begin{enumerate}
\item
 $w(bX)=w(X)$,
\item
$|H|\leq w(X)$ and $w(H)\leq w(X)$,
\item
$\gamma$ is a $d$-open action,
\item components of actions $\alpha$ and $\gamma$ coincide.
\end{enumerate}
\end{thm}

\begin{rem}
{\rm The group $H$ in Theorem~\ref{t3.2} can be considered as a
subgroup of $\Hom(bX)$ with the compact open topology.

In~\cite{Ar} it is proved that the compact open topology is the
least admissible topology on the group $\Hom(bX)$. In~\cite{SM} it
is noted that the compact open topology is the least admissible
topology on any subgroup $H$ of $\Hom(bX)$. In fact, let $K$ be a
compact and $W$ an open subset of $bX$, $[K, W]=\{h\in H \mid
h(K)\subset W\}$ is an element of a sub-base of the compact open
topology. Fix $h\in [K, W]$. For any $x\in K$ there are a nbd $Ox$
of $x$ and a nbd $V_xh$ of $h$ (in the original topology) such that
$\gamma (V_xh, Ox)\subset W$. Take a finite set of nbds $Ox_1,\dots,
Ox_n$ such that $K\subset O=Ox_1\cup\dots\cup Ox_n$, and denote
$V:=V_{x_1}\cap\dots\cap V_{x_n}$. Then $\gamma (Vh, O)\subset W$.
This shows that the set $[K, W]$ is open in $H$ (in the original
topology).

Since the restriction $\gamma|_{H\times X}: H\times X\to X$ is
$d$-open, it will be, evidently, $d$-open when $H$ is equipped with a
weaker topology.}
\end{rem}

\begin{rem}\label{r3.2}
{\rm From the proof of Theorem~\ref{t3.1} it is easy to see that any
subgroup (not topological) of $G$ of cardinality $w(X)$ can be a
subgroup of the group $H$ in Theorems~\ref{t3.1} and
Theorem~\ref{t3.2}.}
\end{rem}

\begin{que}
Can Theorem~\ref{t3.2}  be strengthened in such a way that the action $\gamma: H\times bX\to bX$  is  $d$-open?
\end{que}

Is it possible to be strengthened in the following way.

\begin{que} Can Theorem~\ref{t3.2}  be strengthened in such a way that the action $\gamma: H\times X\to X$  is  $d$-open and  its
extension to the completion of $X$ with respect to the maximal equiuniformity is
$d$-open too?
\end{que}

The following corollary gives an answer to Question 2.6
from~\cite{K2016}.

\begin{cor}\label{c3.1}
Let $(G, X, \alpha)$ be a $G$-space with a $d$-open action, where
$X$ is a separable metrizable space. Then there exist a countable
metrizable group $H$,  a metrizable compactification $bX$ of $X$ and
a continuous action $\gamma: H\times bX\to bX$   such that the
restriction $\gamma_{H\times X} : H\times X\to X$ is a  $d$-open
action. Moreover, if we fix an arbitrary countable number of
homeomorphisms $S$ from $G$, then $H$ can be constructed in such a
way that $S\subset H$.
\end{cor}

\begin{cor}\label{c3.3}
Let $(G, X, \alpha)$ be a $G$-space with a $d$-open action, where
$X$ is a separable metrizable space. Then there exist a  Polish
$G$-space $(T, Y, \delta)$ with an open action, a  countable
subgroup $H'$ of $G$   such that  $(H', X, \alpha|_{H'\times X})$ is
equivariantly embedded in $(T, Y, \delta)$ and the image of $H'$ is
dense in $T$.

\end{cor}

\begin{proof}
From the proof of Theorem~\ref{t3.1} it follows that there exists a
countable subgroup $H'$ of $G$ such that $(q,\id):(H', X,
\alpha|_{H'\times X})\rightarrow (H, X, \gamma)$ is an equivariant
embedding in a separable metrizable  topological transformation
group $(H, X, \gamma)$ with a $d$-open action. From
Corollary~\ref{c2.9} it follows that $(H, X, \gamma)$ has an
equivariant compactification $(H, bX, \tilde\gamma)$,  where $bX$ is
metrizable. Let $Y\subset bX$ be the set of points at which the
action $\tilde\gamma$ is $d$-open. It is an invariant subset of $bX$
and $X\subset Y$ by~\cite[Proposition 3]{K4}.

By~\cite{Mg1} $(H, bX, \tilde\gamma_\rho)$ is equivariantly embedded
in $(\rho H, bX, \tilde\gamma)$ and the action of $\rho H$ on $bX$
is $d$-open at the same invariant set of points $Y$ by
Lemma~\ref{lA.5}. Put $T=\rho H$, $T$ is Polish. In~\cite{An} it is
proved that the action of $T$ on $Y$ is open and, hence, $Y$ is
Polish.
\end{proof}

\begin{rem}\label{r3.3}
{\rm In order to prove Corollaries~\ref{c3.1} and~\ref{c3.3} we can
use M.~Megrelishvili's Theorem 2.5 from~\cite{Mg4}. We can
equivariantly embed a separable metrizable $G$-space in $(I^\omega,
\Hom (I^\omega), \pi)$ and then take the closures of the embedded
group in $\Hom (I^\omega)$ with the compact open topology and the
embedded space in $I^\omega$.}
\end{rem}

\begin{cor}\label{c3.4}
Every separable metrizable space $X$ which is a coset space of some
group has a Polish extension $\tilde X$ which is a coset space of a
Polish group. Moreover, if we fix any countable number of
homeomorphisms $S$ of $X$ then $\tilde X$ can be constructed in such
a way that $S$ are extended to $\tilde X$.
\end{cor}


\end{document}